\documentclass[11pt]{amsart}

\usepackage[T1]{fontenc}
\usepackage[a4paper,margin=1in]{geometry}
\usepackage{amsmath,amssymb,amsthm}
\usepackage{enumitem}
\usepackage[colorlinks=true,linkcolor=blue,citecolor=blue,urlcolor=blue]{hyperref}
\usepackage{orcidlink}
\usepackage{tikz-cd}
\usepackage{todonotes}

\usepackage{mathtools}
\mathtoolsset{showonlyrefs}

\newtheorem{theorem}{Theorem}[section]
\newtheorem{proposition}[theorem]{Proposition}
\newtheorem{lemma}[theorem]{Lemma}
\newtheorem{corollary}[theorem]{Corollary}
\theoremstyle{definition}
\newtheorem{definition}[theorem]{Definition}

\theoremstyle{remark}
\newtheorem{remark}[theorem]{Remark}

\newcommand{\Dist}{\operatorname{Dist}}

\newcommand{\Homeo}{\operatorname{Homeo}}
\newcommand{\MEF}{\operatorname{MEF}}

\newcommand{\T}{\mathbb T}
\newcommand{\Z}{\mathbb Z}
\newcommand{\R}{\mathbb R}
\newcommand{\cZ}{\mathcal Z}

\newcommand{\e}{\mathrm e}

\title
[Joint point-distality and structure theory]
{Joint Point-Distality and structure theory for commuting actions}
\author[E. Glasner]{Eli Glasner}
\address[E. Glasner]{Department of Mathematics, Tel Aviv University, Tel Aviv, Israel}

\email{glasner@math.tau.ac.il}
\author[C. Liu]{Chunlin Liu}
\address[C. Liu]{School of Mathematical Sciences, Dalian University of Technology, Dalian, 116024, P.R. China and Institute of Mathematics, Polish Academy of Sciences, ul. Śniadeckich 8, 00-656 Warszawa, Poland}

\email{chunlinliu@mail.ustc.edu.cn, chunlin.liu.math@gmail.com}
\thanks{This paper was supported by   the
	Postdoctoral Fellowship Program and China Postdoctoral
	Science Foundation under Grant Number BX20250067, and the China Postdoctoral Science
	Foundation under Grant Number 2025M773074.}

\begin{document}

\begin{abstract}
Let $G$ and $H$ act commutatively and minimally on a compact metrizable
space $X$.  We prove that if the two subactions are point-distal, equivalently
HPI, then the action generated by them is again point-distal.  We also
construct a common highly proximal extension on which both subactions are
strictly HPI.  As consequences, transitivity of mixed products and linear
iterates upgrades to minimality in the point-distal category.  We then
compare the canonical structure of the two subactions with that of the joint
action.  The maximal highly proximal operations coincide, whereas the
maximal joint isometric factor over a common factor is the meet of the two
subaction-isometric factors.  This gives a recursive description of the
joint Furstenberg tower.  Finally, we construct commuting minimal distal
homeomorphisms of $\mathbb T^3$ whose canonical Furstenberg towers are
different, showing that the meet formula can be strict.
\end{abstract}

\maketitle

\section{Introduction and main results}
\label{sec:introduction}

Let $G$ and $H$ act by commuting homeomorphisms on a compact metrizable
space $X$, and put
$
   \Gamma:=\langle G,H\rangle .
$
The purpose of this paper is to study how the classical structure theory of
minimal flows behaves when two minimal actions coexist and commute on the
same phase space. There are two related, but fundamentally different,
questions. The first is an \emph{existence problem}: if the two subactions
belong to a rigid structural class, does the action generated by them remain
in that class? The second is a \emph{canonical compatibility problem}: if
each subaction carries its own canonical structure tower, to what extent are
these towers forced to agree?

These questions arise naturally from the classical structure theory of
minimal topological dynamical systems. A fundamental starting point is
Furstenberg's structure theorem for distal flows
\cite{Furstenberg1963}. In the compact metrizable setting, a minimal distal
flow can be constructed from the trivial flow through a transfinite
succession of isometric extensions, together with inverse limits at limit
ordinals. By taking at each successor stage the maximal available isometric
extension, one obtains the canonical Furstenberg tower. In this way, the
global structure of a distal system is resolved into a hierarchy of relative
equicontinuous, or equivalently relative isometric, extensions.

The point-distal theory developed subsequently enlarged this picture
considerably. Veech's work on point-distal flows and the equicontinuous
structure relation \cite{VeechPointDistal} showed that the distal structure
theorem admits a substantially broader counterpart in which isometric
extensions are combined with almost one-to-one extensions. Ellis further
developed the Veech structure theorem and its consequences for point-distal
minimal flows \cite{EllisVeech}. These ideas were placed in a more general
proximal--isometric framework by Ellis, Glasner and Shapiro \cite{EGS}.
Highly proximal extensions were developed by Auslander and Glasner
\cite{AuslanderGlasner}, while the PI and HPI structure theories were
further developed in 
\cite{McMahonNachman,VanDerWoude}.

A point-distal system is a highly proximal factor of a strictly HPI
system, where a strictly HPI system is obtained from the one-point flow by
a transfinite tower whose successor maps are either highly proximal or
isometric and whose limit stages are inverse limits. As in the distal case,
one may choose canonical maximal operations at the successive stages. This
canonical structure separates the two basic mechanisms occurring in
point-distal dynamics: highly proximal extensions, which are closely related
to almost one-to-one maps, and isometric extensions, which encode the
relative equicontinuous structure.

A common feature of the classical structure theory, however, is that the
acting group is fixed throughout the construction. Proximality,
equicontinuity, isometricity, and the associated maximal intermediate
factors are all defined relative to that action. Once the same compact space
carries two commuting minimal actions, a new compatibility problem appears.
Suppose that both $(X,G)$ and $(X,H)$ are distal, point-distal, or HPI.
Does the action generated by them have the same structural property? Even
if it does, do the two canonical structure towers agree?

There are also reasons to expect substantial canonical rigidity.  For
example, Shao and Xu \cite{ShaoXu} showed that commuting minimal
homeomorphisms share their higher-order regionally proximal relations and
the associated maximal pro-nilfactors.
On the other hand, the Furstenberg--HPI construction is built from two
primitive operations which behave quite differently when the acting group
is enlarged.
High proximality is
equivalent to irreducibility of the underlying factor map and is therefore
essentially independent of the acting group. Relative isometricity is
different: it requires invariance of a relative metric under the given
action and is consequently sensitive to which group is acting. This
asymmetry is one of the main themes of the present paper.

A second motivation comes from questions of joint transitivity and joint
minimality.  The study of simultaneous orbit density for several iterates
has a substantial history in topological dynamics.  In the weakly mixing
minimal setting, Glasner \cite{Glasner1994} proved joint transitivity for
distinct nonzero powers of a single transformation.  This result was later
extended by Huang, Shao and Ye \cite{HuangShaoYe}, who treated polynomial
expressions arising from several transformations under suitable weak
mixing assumptions, including nilpotent group actions.

A general criterion for linear iterates was later obtained by
Donoso--Koutsogiannis--Sun \cite{DKS}.  For a minimal
$\mathbb Z^k$-system and commuting transformations
$T_1,\ldots,T_d$, they proved that the families
$(T_1^n)_n,\ldots,(T_d^n)_n$ are jointly transitive if and only if
$T_1\times\cdots\times T_d$ is transitive and every pairwise difference
$T_i^{-1}T_j$, $i\neq j$, is transitive.  Their theorem gives a
topological counterpart of the classical Berend--Bergelson criterion for
joint ergodicity.

More recently, and independently of the work of
Donoso--Koutsogiannis--Sun \cite{DKS}, Glasner \cite{GlasnerJoint}
studied joint topological transitivity and the stronger notion of joint
minimality for commuting minimal homeomorphisms.  In particular, he
investigated to what extent the Berend--Bergelson type criterion continues
to hold with transitivity replaced by minimality, and this led naturally to
structural questions in the distal setting.    The present paper develops this structural direction.  We show that the
joint-distality conclusion not only survives but extends to the much broader
point-distal, or equivalently HPI, category.  By contrast, the stronger
expectation that the two subactions should share the same canonical
Furstenberg tower is false; the correct compatibility statement is instead
given by a local comparison of their canonical highly proximal and
isometric successors.

We begin with the existence problem. Our first main result shows that
point-distality is stable under commuting generation.

\begin{theorem}
\label{thm:intro-joint}
Let $G$ and $H$ be countable discrete groups acting commutatively on a
compact metrizable space $X$. If both $(X,G)$ and $(X,H)$ are
minimal point-distal, then the generated system $(X,\Gamma)$ is minimal
point-distal. Equivalently, if both subactions are HPI, then the joint
action is HPI.
\end{theorem}

For a general HPI action an additional difficulty appears. A chosen strict
HPI cover of the $G$-action need not admit a lift of the commuting
$H$-action, so one cannot simply pass to an arbitrary strictification and
apply the strict case. We overcome this obstruction by constructing an
equivariant orbit-lift space. This gives a common structural model for the
two subactions and leads to the following simultaneous strictification
theorem.

\begin{theorem}
\label{thm:intro-strictification}
Under the hypotheses of Theorem~\ref{thm:intro-joint}, there exist a compact
metrizable space $\widehat X$, commuting minimal lifts of the $G$- and
$H$-actions, and a highly proximal joint factor map
$
   p:(\widehat X,\Gamma)\to(X,\Gamma)
$
such that both $(\widehat X,G)$ and $(\widehat X,H)$ are strictly HPI.
Moreover, $(\widehat X,\Gamma)$ is minimal HPI.
\end{theorem}

Thus commuting point-distal actions always admit a common strict HPI model,
although such a common model is not supplied by the individual structure
theorems. It remains open whether the strictification is unnecessary,
namely whether the joint action itself must be strictly HPI whenever both
subactions are minimal strictly HPI.

Theorem~\ref{thm:intro-joint} also gives a common mechanism for the
orbit-density questions discussed above. If $(X,\Gamma)$ is minimal
point-distal and a subgroup $\Lambda\leq\Gamma$ acts transitively, then
the $\Lambda$-action is already minimal point-distal. Likewise, if
$R_1,\ldots,R_d\in\Gamma$ and
$
   R_1\times\cdots\times R_d
$
is transitive on $X^d$, then this product system is minimal point-distal.
When the $R_i$ commute, every difference $R_i^{-1}R_j$ is then minimal
point-distal. Combining this with the criterion of
Donoso--Koutsogiannis--Sun \cite{DKS}, we obtain that in a minimal
point-distal $\mathbb Z^k$-system the following conditions are equivalent:
product transitivity of $T_1\times\cdots\times T_d$, minimality of this
product, joint transitivity of the families
$(T_1^n)_n,\ldots,(T_d^n)_n$, and joint minimality of these families.

\medskip

We next turn from existence of a joint structure to the substantially finer
problem of canonical compatibility. The distinction is important. The
fact that $(X,\Gamma)$ is HPI guarantees the existence of an HPI
presentation for the joint action, but it does not imply that the canonical
presentations of the two subactions agree.

Let
$
   q:(X,G,H)\to(Y,G,H)
$
be a common factor map. For an acting group $L$, write
$\mathsf I_L(q)$ for the maximal $L$-isometric intermediate factor of
$q$, and $\mathsf H_L(q)$ for the canonical maximal $L$-highly
proximal intermediate factor. Instead of comparing whole transfinite
towers at once, we first compare these two primitive canonical operations
over an arbitrary common factor.

\begin{theorem}
\label{thm:intro-primitive-comparison}
Let $G$ and $H$ act commutatively and minimally on $X$, and put
$\Gamma=\langle G,H\rangle$. Then the $G$-, $H$-, and
$\Gamma$-actions have the same maximal equicontinuous factor. Moreover,
for every common factor $q:X\to Y$,
\[
   \mathsf H_G(q)
   =
   \mathsf H_H(q)
   =
   \mathsf H_\Gamma(q),
\]
whereas
\[
   \mathsf I_\Gamma(q)
   =
   \mathsf I_G(q)\wedge_Y\mathsf I_H(q).
\]
\end{theorem}

Consequently, if the canonical HPI constructions for $G$ and $H$ agree
up to a common stage, then their next highly proximal successors agree as
well. In particular, the first discrepancy between the two canonical
towers can occur only at an isometric stage. At such a stage the canonical
successor for the joint action is the meet of the two subaction successors.

In the distal case, Theorem~\ref{thm:intro-primitive-comparison} gives a
particularly transparent canonical description. If both $(X,G)$ and
$(X,H)$ are minimal distal, then $(X,\Gamma)$ is minimal distal.
Starting from their common maximal equicontinuous factor, the canonical
Furstenberg tower of the joint action is obtained recursively by replacing
each successor with
$
   \mathsf I_G(q_\alpha)
   \wedge_{Y_\alpha}
   \mathsf I_H(q_\alpha)
$
and taking inverse limits at limit ordinals.

It is natural at this point to ask for a stronger rigidity statement:
perhaps commuting minimal distal actions have the same canonical
Furstenberg tower. The preceding meet formula shows what such a statement
would require, but does not force it. Our final main result shows that the
stronger statement is false, already for commuting minimal distal
homeomorphisms of a finite-dimensional torus.

\begin{theorem}
\label{thm:intro-counterexample}
There exist commuting minimal distal homeomorphisms
$T,S\in\Homeo(\mathbb T^3)$, with
$\Gamma=\langle T,S\rangle$, and common factors
\[
   X=\mathbb T^3
   \xrightarrow{\rho}
   Z=\mathbb T^2
   \xrightarrow{\pi_0}
   Y=\mathbb T
\]
such that $Y$ is the maximal equicontinuous factor of the $T$-,
$S$-, and $\Gamma$-actions, but
\[
   \mathsf I_{\langle T\rangle}(Y)=X,
   \qquad
   \mathsf I_{\langle S\rangle}(Y)
   =
   \mathsf I_\Gamma(Y)
   =
   Z.
\]
Consequently, the canonical Furstenberg tower of $T$ is
\[
   \{*\}\longleftarrow Y\longleftarrow X,
\]
whereas the canonical towers of $S$ and $\Gamma$ are
\[
   \{*\}\longleftarrow Y\longleftarrow Z\longleftarrow X.
\]
\end{theorem}

This also highlights a useful contrast with the higher-order regionally
proximal theory. By the result of Shao and Xu \cite{ShaoXu}, the
higher-order regionally proximal relations and the corresponding maximal
pro-nilfactors are invariant under passage between commuting minimal
homeomorphisms. The canonical Furstenberg tower is more sensitive. Its
highly proximal stages enjoy an analogous rigidity, but its isometric
stages need not coincide. 

The paper is organized according to this progression.
Section~\ref{sec:preliminaries} collects the HPI and relative-isometric
background needed later.
Section~\ref{sec:joint-point-distality} proves Theorems \ref{thm:intro-joint} and \ref{thm:intro-strictification}.
Section~\ref{sec:applications} records the resulting transitivity and
linear-iterate consequences.
Section~\ref{sec:canonical} proves Theorem \ref{thm:intro-primitive-comparison}.
Finally, Section~\ref{sec:counterexample} constructs the distal torus
example showing Theorem \ref{thm:intro-counterexample}.

\section{Preliminaries}
\label{sec:preliminaries}

Throughout this paper, all acting
groups are countable and discrete, all phase spaces are compact metrizable,
and all actions are by homeomorphisms.

Let a group $L$ act on a compact metric space $(X,d)$, and call 
$(X,L)$ a \emph{flow}.  A pair $(x,x')\in X^2$ is
\emph{$L$-proximal} if there is a sequence $\ell_n\in L$ such that
\[
 d(\ell_nx,\ell_nx')\longrightarrow0.
\]
A point $x$ is \emph{$L$-distal} if the only point $L$-proximal to $x$ is
$x$ itself.  A minimal $L$-system is \emph{point-distal} if it has an
$L$-distal point.  We write $\Dist_L(X)$ for the set of $L$-distal points.

For an extension $\pi:(X,L)\to(Y,L)$, put
\[
 R_\pi:=\{(x,x')\in X^2:\pi(x)=\pi(x')\}.
\]
The extension is \emph{$L$-isometric} if there is a continuous relative
metric
$
 \rho:R_\pi\to[0,\infty)
$
whose restriction to every fibre is a compatible metric and which satisfies
\[
 \rho(\ell x,\ell x')=\rho(x,x')
 \qquad(\ell\in L,\ (x,x')\in R_\pi).
\]
It is \emph{highly proximal} if for any nonempty open subset $U\subset X$, there exists $y\in Y$ and $g\in L$ such that $g\pi^{-1}(y)\subset U$.

A minimal flow is \emph{strictly HPI} if it is obtained from the one-point
flow by a transfinite tower whose successor maps are either highly proximal
or isometric and whose limit stages are inverse limits.  It is \emph{HPI} if
it is a highly proximal factor of a strictly HPI flow.  In the metrizable
category, highly proximal extensions are exactly 
almost one-to-one. 
Consequently, in the metrizable setting, the HPI property may equivalently be formulated using almost one-to-one and isometric extensions. It is then referred to as having the AI property.

We use the following standard facts from the structure theory of Veech,
Ellis--Glasner--Shapiro, Auslander--Glasner, and van der Woude
\cite{VeechPointDistal,EGS,AuslanderGlasner,VanDerWoude}.

\begin{proposition}\label{prop:standard-HPI}
Let $(X,L)$ be a  minimal flow.
\begin{enumerate}[label=\textup{(\roman*)}]
\item $(X,L)$ is HPI if and only if it is point-distal.
\item If $(X,L)$ is point-distal, then $\Dist_L(X)$ is a dense
      $G_\delta$ subset of $X$.
\item Every metrizable factor of an HPI flow is HPI.
\item Every factor map whose domain is minimal is semi-open.
\item Every distal point of a compact flow is a minimal point.
\end{enumerate}
Moreover, a metrizable strictly HPI flow has a canonical HPI tower of
countable ordinal height.
\end{proposition}

Recall that a factor map $\pi\colon X\to Y$ is \emph{semi-open} if the image of every nonempty open subset of $X$ has nonempty interior in $Y$.
We shall repeatedly use the following elementary consequences.

\begin{lemma}\label{lem:residual-pullback}
Let $f:W\to Y$ be a continuous semi-open surjection between compact metric
spaces.  If $A\subseteq Y$ is a dense $G_\delta$ set, then $f^{-1}(A)$ is a
dense $G_\delta$ subset of $W$.
\end{lemma}

\begin{proof}
Write $A=\bigcap_{n\geq1}V_n$ with every $V_n$ open and dense.  If
$U\subseteq W$ is nonempty and open, then $f(U)$ has nonempty interior, so
$V_n\cap\operatorname{int}f(U)\neq\varnothing$.  Hence
$U\cap f^{-1}(V_n)\neq\varnothing$.  Thus every $f^{-1}(V_n)$ is open and
dense, and the conclusion follows from the Baire theorem.
\end{proof}

A continuous surjection $\pi\colon X\to Y$ between compact Hausdorff spaces is
called \emph{irreducible} if no proper closed subset of $X$ maps onto $Y$.
Notice that irreducibility is a purely topological property of the underlying
continuous surjection $\pi$, rather than a dynamical property of the actions
on $X$ and $Y$.
\begin{proposition}[Auslander--Glasner {\cite[p.~733]{AuslanderGlasner}}]\label{prop:hp-irreducible}
An extension $\pi$ of minimal flows is highly proximal if and only if it is irreducible.
\end{proposition}

\begin{lemma}\label{lem:singleton-fibres}
Let $\pi:W\to Y$ be an irreducible continuous surjection between compact
metric spaces.  Then
\[
 Y_\pi:=\{y\in Y:|\pi^{-1}(y)|=1\}
\]
is a dense $G_\delta$ subset of $Y$.
\end{lemma}

\begin{proof}
Fix a compatible metric on $W$ and put
\[
 U_n:=\{y\in Y:\operatorname{diam}\pi^{-1}(y)<1/n\}.
\]
Upper semicontinuity of the fibre map shows that $U_n$ is open.  Let
$V\subseteq Y$ be nonempty and open.  Choose a nonempty open set
$O\subseteq\pi^{-1}(V)$ of diameter less than $1/n$.  Irreducibility implies
that some full fibre is contained in $O$, and its image lies in $V\cap U_n$.
Thus every $U_n$ is dense, and $Y_\pi=\bigcap_nU_n$.
\end{proof}

\begin{lemma}\label{lem:aoto-lift}
Let $\pi:(W,L)\to(Y,L)$ be an almost one-to-one extension between minimal
systems.  If $(Y,L)$ is point-distal, then $(W,L)$ is point-distal.
\end{lemma}

\begin{proof}
Choose
\[
 y\in\Dist_L(Y)\cap Y_\pi,
 \qquad \pi^{-1}(y)=\{w\}.
\]
If $w'$ is $L$-proximal to $w$, then $\pi(w')$ is $L$-proximal to $y$;
hence $\pi(w')=y$ and therefore $w'=w$.
\end{proof}

 For the lifting
arguments and the later comparison of canonical towers, we also need the
relative equicontinuous structure associated with a factor map.

Let $q:(X,L)\to(Y,L)$ be an extension of minimal systems.  Its
\emph{relative regionally proximal relation} $Q_L(q)$ consists of the pairs
$(x,x')\in R_q$ for which there exist sequences $x_n\to x$, $x_n'\to x'$,
and $\ell_n\in L$ such that
\[
 q(x_n)=q(x_n')
 \quad\text{and}\quad
 d(\ell_nx_n,\ell_nx_n')\longrightarrow0.
\]
The \emph{relative equicontinuous structure relation} $E_L(q)$ is the
smallest closed equivalence relation containing $Q_L(q)$ and invariant under
the diagonal $L$-action on $X^2$.

An extension
$
 \pi:(Z,L)\to(Y,L)
$
is \emph{relatively equicontinuous} if, for every entourage $U$ of the
diagonal in $Z^2$, there exists an entourage $V$ such that
\[
 (z,z')\in V\cap R_\pi
 \quad\Longrightarrow\quad
 (\ell z,\ell z')\in U
 \quad\text{for every }\ell\in L.
\]
For extensions of compact metrizable minimal systems, relative
equicontinuity is equivalent to isometricity; see
\cite[Chapter~9]{Auslander}.

The following is standard; see \cite[p.~324]{AuslanderMcMahonWoudeWu}.
\begin{proposition}
\label{prop:maximal-isometric}
Let
$
q:(X,L)\to (Y,L)
$
be an extension of compact metrizable minimal systems. Then there exists,
uniquely up to isomorphism over $Y$, a largest $L$-isometric intermediate
factor
\[
X\xrightarrow{p_{\rm iso}} Z_{\rm iso}
 \xrightarrow{\pi_{\rm iso}}Y,
\qquad
q=\pi_{\rm iso}\circ p_{\rm iso},
\]
and
\[
Z_{\rm iso}=X/E_L(q),
\]
where $E_L(q)$ is the relative equicontinuous structure relation of $q$.

\end{proposition}

The next compactness observation will be used at every isometric lifting
stage.

\begin{lemma}\label{lem:fibre-metric}
Let $\pi:W\to Y$ be a continuous surjection, let $d$ be a compatible metric
on $W$, and let $\rho:R_\pi\to[0,\infty)$ be continuous and restrict to a
metric on every fibre.  For every sequence $(u_n,v_n)\in R_\pi$,
\[
 d(u_n,v_n)\longrightarrow0
 \quad\Longleftrightarrow\quad
 \rho(u_n,v_n)\longrightarrow0.
\]
\end{lemma}

\begin{proof}
If $d(u_n,v_n)\to0$, uniform continuity of $\rho$ on the compact set $R_\pi$
and 
$\rho(u_n,u_n) \to 0$ give $\rho(u_n,v_n)\to0$.  Conversely, if
$\rho(u_n,v_n)\to0$ but $d(u_n,v_n)\not\to0$, a convergent subsequence in
$R_\pi$ 
tends to a distinct pair on which $\rho$ vanishes, contradicting
that $\rho$ is a metric on each fibre.
\end{proof}

\section{Joint point-distality}\label{sec:joint-point-distality}
In this section, we  prove Theorems \ref{thm:intro-joint} and \ref{thm:intro-strictification}.  The main issue is that the HPI structure theorem is
formulated for a single acting group: a canonical HPI tower for the
$G$-action is not, a priori, a tower of factors for the commuting
$H$-action.  We first show that its canonical successor factors are
preserved by the $H$-action.  This makes the whole canonical $G$-HPI
tower into a tower of common $G$- and $H$-factors.
Once this has been established, joint distal points can be propagated
through the tower. 

\subsection{Characteristic stages and primitive lifting}
We begin with the functoriality of the two maximal intermediate factors under
commuting actions.

\begin{proposition}
\label{prop:characteristic-successors}
Let
$
 q:(X,\Gamma)\to(Y,\Gamma)
$
be a factor map.  Then the maximal $G$-isometric intermediate factor
$\mathsf I_G(q)$ and the canonical maximal $G$-highly proximal intermediate
factor $\mathsf H_G(q)$ are invariant under the $H$-action.  In particular,
they are common $\Gamma$-factors.

Symmetrically, $\mathsf I_H(q)$ and $\mathsf H_H(q)$ are invariant under the
$G$-action and hence are common $\Gamma$-factors.
\end{proposition}

\begin{proof}
We first consider the maximal isometric factor.  Let $h\in H$.  Since the
$G$- and $H$-actions commute and $q$ is $H$-equivariant,
\[
 (h\times h)Q_G(q)=Q_G(q).
\]
Indeed, suppose that $x_n\to x$, $x_n'\to x'$, and $g_n\in G$ witness
$(x,x')\in Q_G(q)$.  Then
\[
 q(hx_n)=h q(x_n)=h q(x_n')=q(hx_n'),
\]
and, by commutativity and uniform continuity of the fixed homeomorphism $h$,
\[
 d(g_nhx_n,g_nhx_n')
 =d(hg_nx_n,hg_nx_n')\longrightarrow0.
\]
Applying the same argument to $h^{-1}$ gives the reverse inclusion.  Hence
\[
 (h\times h)E_G(q)=E_G(q).
\]
Since
$
 \mathsf I_G(q)=X/E_G(q),
$
the homeomorphism $h$ descends to $\mathsf I_G(q)$.

We next consider the canonical maximal highly proximal intermediate factor.
Write
\[
 X\xrightarrow{p}Z\xrightarrow{\pi}Y,
 \qquad q=\pi p,
\]
where $Z=\mathsf H_G(q)$ and $\pi$ is highly proximal.  For $h\in H$, put
\[
 R_p^h:=(h\times h)R_p,
 \qquad Z^h:=X/R_p^h.
\]
The map
\[
 \Phi_h:Z\longrightarrow Z^h,
 \qquad \Phi_h(p(x))=[hx]_{R_p^h},
\]
is a homeomorphism.  The induced map $\pi^h:Z^h\to Y$ satisfies
\[
 \pi^h\Phi_h=h\pi.
\]
Thus $\pi^h$ is highly proximal, because highly proximality is equivalent to
irreducibility and irreducibility is invariant under topological conjugacy;
see Proposition~\ref{prop:hp-irreducible}.

It follows that $h$ permutes the highly proximal intermediate factors of $q$.
By uniqueness and maximality of the canonical maximal member,
\[
 (h\times h)R_p=R_p.
\]
Therefore $h$ descends to $Z=\mathsf H_G(q)$.  The symmetric assertions follow
by interchanging $G$ and $H$.
\end{proof}

We may apply 
the proposition above to the canonical HPI tower.
\begin{corollary}
\label{cor:canonical-characteristic}
Let $(X,G)$ be a minimal strictly HPI system, and let
\begin{equation}\label{eq:canonical-tower}
 \{*\}=X_0
 \longleftarrow X_1
 \longleftarrow\cdots
 \longleftarrow X_\alpha
 \longleftarrow\cdots
 \longleftarrow X_\eta=X
\end{equation}
be its canonical HPI tower. If a group $H$ acts on $X$ and commutes with
the $G$-action, then the $H$-action descends to every $X_\alpha$.
Moreover, every bonding map in \eqref{eq:canonical-tower} is
$\Gamma$-equivariant, and the induced $G$- and $H$-actions commute
at every stage.
\end{corollary}

\begin{proof}
We proceed by transfinite induction. The assertion is trivial at the
one-point stage. Suppose that the $H$-action has descended to $X_\alpha$
and let
$
 q_\alpha:X\to X_\alpha
$
be the canonical factor map.

At a successor stage, the next factor is either the maximal
$G$-isometric intermediate factor of $q_\alpha$ or the canonical maximal
$G$-highly proximal intermediate factor. Proposition
\ref{prop:characteristic-successors} shows in either case that this
factor is invariant under the $H$-action.

At a limit ordinal, the compatible $H$-actions on the preceding stages
induce an action on their inverse limit. The induced action commutes with
the $G$-action because the same is true at every preceding stage.
\end{proof}

The characteristic property allows the second action to be followed along
the entire canonical tower.  We next record the pointwise strengthening
needed to lift joint distality through its two primitive successor maps.  The
following lemma treats the isometric case and will also be used in the strict
HPI argument.
\begin{lemma}\label{lem:mixed-isometric}
Let $G$ and $H$ act commutatively on compact metrizable spaces $W$ and $Y$, let $\Gamma:=\langle G,H\rangle$, 
and let
$
 \pi:(W,\Gamma)\to(Y,\Gamma)
$
be an equivariant extension.  Suppose that $\pi$ is $G$-isometric.  Let
$w\in W$ and $y=\pi(w)$.  If $y$ is $\Gamma$-distal and $w$ is
$H$-distal, then $w$ is $\Gamma$-distal.
\end{lemma}

\begin{proof}
Let $\rho$ be a continuous $G$-invariant relative metric on $R_\pi$, and
suppose that $w'$ is $\Gamma$-proximal to $w$.  There are sequences
$g_n\in G$ and $h_n\in H$ such that
\begin{equation}\label{eq:2026826}
    \lim_{n\to\infty}d(g_nh_nw,g_nh_nw')=0.
\end{equation}
 After projection to $Y$, the pair $(y,\pi(w'))$ is jointly proximal.  Since
$y$ is jointly distal, $\pi(w')=y$.  Thus
$(h_nw,h_nw')\in R_\pi$ for every $n$, and
Lemma~\ref{lem:fibre-metric}, followed by $G$-invariance and \eqref{eq:2026826}, gives
\[
 \rho(h_nw,h_nw')
 =\rho(g_nh_nw,g_nh_nw')\to0,\qquad\text{as } n\to\infty.
\]
A second application of Lemma~\ref{lem:fibre-metric} yields
$d(h_nw,h_nw')\to0$.  Hence $w'$ is $H$-proximal to $w$, and the
$H$-distality of $w$ gives $w'=w$.
\end{proof}

The mixed isometric lifting lemma above is the local mechanism at an
isometric successor stage.  At a highly proximal successor we use the
following elementary observation.

\begin{lemma}
\label{lem:singleton-lifting}
Let $\pi:(W,L)\to(Y,L)$ be an extension.  Suppose that $y\in Y$ is
$L$-distal and that $\pi^{-1}(y)=\{w\}$.  Then $w$ is $L$-distal.
\end{lemma}

\begin{proof}
If $w'$ is $L$-proximal to $w$, then $\pi(w')$ is $L$-proximal to $y$.
Thus $\pi(w')=y$, and the singleton-fibre assumption gives $w'=w$.
\end{proof}
\subsection{The strict case and equivariant strictification}

We now combine the characteristic-stage result with the two lifting lemmas.
The strict case is the inductive core of the proof.

\begin{theorem}
\label{thm:strict-case}
Let $G$ and $H$ be groups acting commutatively on a compact metrizable space
$X$, and put $\Gamma:=\langle G,H\rangle$.  Assume that $(X,G)$ is minimal
strictly HPI and that $(X,H)$ is minimal HPI.  Then the set of
$\Gamma$-distal points contains a dense $G_\delta$ subset of $X$.  In
particular, the joint action is minimal point-distal, and hence HPI.
\end{theorem}

\begin{proof}
Let
\[
 \{*\}=X_0\longleftarrow X_1\longleftarrow\cdots
 \longleftarrow X_\alpha\longleftarrow\cdots
 \longleftarrow X_\eta=X
\]
be the canonical $G$-HPI tower, and let $q_\alpha:X\to X_\alpha$ be the
canonical factor maps.  By Corollary~\ref{cor:canonical-characteristic}, the
$H$-action descends to every stage.

Since HPI is inherited by metrizable factors (Proposition \ref{prop:standard-HPI} (iii)), every $(X_\alpha,H)$ is HPI.
Consequently,
$
 D_\alpha:=\Dist_H(X_\alpha)
$
is a dense $G_\delta$ subset of $X_\alpha$.  If
$\pi_\alpha:X_{\alpha+1}\to X_\alpha$ is a highly proximal successor map,
put
\[
 C_\alpha:=\{y\in X_\alpha:|\pi_\alpha^{-1}(y)|=1\}.
\]
By Proposition~\ref{prop:hp-irreducible} and
Lemma~\ref{lem:singleton-fibres}, $C_\alpha$ is a dense $G_\delta$ subset of
$X_\alpha$.

By Proposition \ref{prop:standard-HPI} (iv), every $q_\alpha$ is semi-open.  Since the tower has countable height,
Lemma~\ref{lem:residual-pullback} and the Baire theorem show that
\begin{equation}\label{eq:strict-residual-choice}
 \mathcal R=
 \bigcap_{\alpha\leq\eta}q_\alpha^{-1}(D_\alpha)
 \cap
 \bigcap_{\substack{\alpha<\eta\\
          \pi_\alpha\text{ highly proximal}}}
 q_\alpha^{-1}(C_\alpha)
\end{equation}
is a dense $G_\delta$ subset of $X$.  Fix $x\in\mathcal R$ and write
$x_\alpha=q_\alpha(x)$.

We prove by transfinite induction that $x_\alpha$ is $\Gamma$-distal.  This is
trivial at $X_0$.  At an isometric successor stage, Lemma~\ref{lem:mixed-isometric} gives that
$x_\alpha$ is $\Gamma$-distal by induction and $x_{\alpha+1}$ is $H$-distal.  At a highly proximal
successor stage, the choice of $x$ gives
$
 \pi_\alpha^{-1}(x_\alpha)=\{x_{\alpha+1}\},
$
and Lemma~\ref{lem:singleton-lifting} applies.  Since inverse
limit coordinates separate points, every point of $\mathcal R$ is
$\Gamma$-distal.  The $\Gamma$-action is minimal because it contains the minimal
$G$-action, and Proposition~\ref{prop:standard-HPI} (i) completes the proof.
\end{proof}
The preceding argument requires the $G$-action itself to be strictly HPI.
To pass from this inductive core to arbitrary HPI actions, we replace a
non-equivariant strict cover by a joint orbit-lift construction.
For a general HPI $G$-system, one may choose a highly proximal extension
\[
 \theta:(\widetilde X,G)\longrightarrow(X,G)
\]
with $(\widetilde X,G)$ strictly HPI, but the commuting $H$-action on $X$
need not lift to $\widetilde X$.  We therefore construct a new joint
extension carrying both actions, on which the lifted $G$-action is strictly
HPI.

\begin{theorem}
\label{thm:equivariant-strictification}
Let $G$ and $H$ act commutatively on a compact metrizable space $X$, assume
that $H$ is countable, and suppose that both actions are minimal HPI.  Let
$
 \theta:(\widetilde X,G)\to(X,G)
$
be a metrizable highly proximal extension with $(\widetilde X,G)$ strictly
HPI.  Then there exist a compact metrizable  system $(Z,\Gamma)$ and
factor maps
\[
 c_e:(Z,G)\longrightarrow(\widetilde X,G),
 \qquad
 p:(Z,\Gamma)\longrightarrow(X,\Gamma),
\]
satisfying $\theta\circ c_e=p$, with the following properties:
\begin{enumerate}[label=\textup{(\roman*)}]
\item $G$ and $H$ actions on $Z$ commute;
\item $p$ is highly proximal;
\item $c_e$ is highly proximal;
\item $(Z,G)$ is minimal strictly HPI;
\item $(Z,H)$ is minimal HPI;
\item $(Z,\Gamma)$ is minimal HPI.
\end{enumerate}
Here we do not assume $H$ acts on $\widetilde X$.
\end{theorem}

\begin{proof}
Let $e$ be the identity of $H$ and define the closed orbit-lift space
\begin{equation}\label{eq:orbit-lift-space}
 \cZ:=\left\{z=(z_h)_{h\in H}\in\widetilde X^H:
       \theta(z_h)=h\theta(z_e)\text{ for every }h\in H\right\}.
\end{equation}
Since $H$ is countable, $\widetilde X^H$ is compact metrizable.  The space
$\cZ$ is nonempty, as for a fixed $x\in X$, we may choose
$z_h\in\theta^{-1}(hx)$ for every $h\in H$.

For $g\in G$ and $a\in H$, define
\begin{equation}\label{eq:orbit-lift-actions}
 (g\cdot z)_h:=gz_h,
 \qquad
 (a\star z)_h:=z_{ha}.
\end{equation}
Commutation of the original actions shows that the $G$-action preserves
\eqref{eq:orbit-lift-space}.  The second formula defines a left $H$-action,
since $a\star(b\star z)=(ab)\star z$, and
\[
 \theta((a\star z)_h)=ha\theta(z_e)
 =h\theta((a\star z)_e).
\]
The two lifted actions commute.

Define
\[
 p:\cZ\longrightarrow X,
 \qquad p(z):=\theta(z_e).
\]
Then $p(g\cdot z)=gp(z)$ and $p(a\star z)=ap(z)$.  Choose a minimal
$\Gamma$-subsystem $Z\subseteq\cZ$.  The image $p(Z)$ is a nonempty closed
$G$-invariant subset of $X$, so minimality of $(X,G)$ gives $p(Z)=X$.
We henceforth restrict $p$ to $Z$. From the construction, we have $G$ and $H$ actions on $Z$ commute.

Put 
\[
 X_0:=\{x\in X:|\theta^{-1}(x)|=1\},
\]
and set
\begin{equation}\label{eq:orbit-lift-singleton-set}
 X_\infty:=\bigcap_{h\in H}h^{-1}X_0.
\end{equation}
This is a dense $G_\delta$ subset of $X$, as $\theta$ is 
a highly proximal extension and hence almost one-to-one, while $H$ is countable.  If $x\in X_\infty$ and
$z,z'\in Z$ satisfy $p(z)=p(z')=x$, then
\[
 \theta(z_h)=hx=\theta(z_h')
 \qquad(h\in H).
\]
Since $hx\in X_0$, one has $z_h=z_h'$ for all $h$, and hence $z=z'$.
Thus $p$ is almost one-to-one and therefore highly proximal.

We next prove minimality of the individual subactions. Since $(Z,\Gamma)$ is minimal, the factor map $p$ is
semi-open, and so $p^{-1}(X_\infty)$ is dense in $Z$.  If $K\subseteq Z$ is a
nonempty closed $H$-invariant set, then $p(K)$ is a nonempty closed
$H$-invariant subset of $X$, and hence $p(K)=X$.  Every singleton fibre over
$X_\infty$ therefore belongs to $K$, so $p^{-1}(X_\infty)\subseteq K$.
Density gives $K=Z$.  Thus $(Z,H)$ is minimal.  The same argument using the
minimality of $(X,G)$ proves that $(Z,G)$ is minimal.  Since
$p:(Z,H)\to(X,H)$ is almost one-to-one and the base is point-distal,
Lemma~\ref{lem:aoto-lift} shows that $(Z,H)$ is HPI.

Finally, let
\[
 c_e:Z\longrightarrow\widetilde X,
 \qquad c_e(z):=z_e.
\]
It is a $G$-factor map. Indeed, its image is a nonempty closed $G$-invariant subset of
the minimal $G$-system $\widetilde X$, and hence is all of $\widetilde X$.
Moreover, $\theta \circ c_e=p$.  By semi-openness of $\theta$ and
Lemma~\ref{lem:residual-pullback}, the set $\theta^{-1}(X_\infty)$ is a dense
$G_\delta$ subset of $\widetilde X$.  If $y\in \theta^{-1}(X_\infty)$  and
$c_e(z)=c_e(z')=y$, then
\[
 \theta(z_h)=h\theta(y)=\theta(z_h')
 \qquad(h\in H).
\]
Since $h\theta(y)\in X_0$, one has $z_h=z_h'$ for all $h$.  Thus $c_e$ is
almost one-to-one and hence highly proximal.  Appending $c_e$ as one further
highly proximal stage to a strict HPI tower for $(\widetilde X,G)$ proves that
$(Z,G)$ is strictly HPI.  Theorem~\ref{thm:strict-case} now shows that
$(Z,\Gamma)$ is minimal HPI.
\end{proof}

The orbit-lift construction now completes both main structural statements
from the introduction.

\begin{proof}[Proof of Theorem~\ref{thm:intro-joint}]
Choose a metrizable highly proximal extension
$
 \theta:(\widetilde X,G)\to(X,G)
$
with $(\widetilde X,G)$ strictly HPI.  Theorem
\ref{thm:equivariant-strictification} produces a joint highly proximal
extension $p:Z\to X$ for which $(Z,G)$ is strictly HPI, $(Z,H)$ is 
HPI
and $(Z,\Gamma)$ is HPI.  
Since HPI is inherited
by metrizable factors, $(X,\Gamma)$ is HPI and hence point-distal.  Minimality
is automatic.
\end{proof}

\begin{proof}[Proof of Theorem~\ref{thm:intro-strictification}]
First apply Theorem~\ref{thm:equivariant-strictification} to obtain
$
 p_1:(X_1,\Gamma)\to(X,\Gamma),
$
where $(X_1,G)$ is strictly HPI and $(X_1,H)$ is HPI.  Apply the same theorem
to $X_1$ with the roles of $G$ and $H$ interchanged.  This gives
$
 p_2:(\widehat X,\Gamma)\to(X_1,\Gamma),
$
where $(\widehat X,H)$ is strictly HPI and $p_2$ is highly proximal.  Since
$p_2$ is also a highly proximal $G$-extension and $(X_1,G)$ is strictly HPI,
appending $p_2$ to a strict $G$-tower shows that $(\widehat X,G)$ is strictly
HPI.  The composition $p_1p_2$ is highly proximal, as the composition of
two irreducible surjections is irreducible.  Theorem~\ref{thm:strict-case}
makes the joint action on $\widehat X$ HPI.
\end{proof}

\section{Transitivity consequences}
\label{sec:applications}

\subsection{Subgroup and mixed-product rigidity}
\label{sec:rigidity}
We now derive several consequences of joint point-distality for
transitive subactions and mixed Cartesian products.  
\begin{lemma}
\label{lem:subgroup-rigidity}
Let $(X,\Gamma)$ be a minimal point-distal system and let
$\Lambda\leq\Gamma$.  If the $\Lambda$-action is topologically transitive,
then it is minimal point-distal.
\end{lemma}

\begin{proof}
The set $\Dist_\Gamma(X)$ is a dense $G_\delta$ set, and the set of
$\Lambda$-transitive points is also a dense $G_\delta$ set.  Choose a point
$x$ in their intersection.  Since $\Lambda\leq\Gamma$, the point $x$ is
$\Lambda$-distal and hence a $\Lambda$-minimal point.  Its $\Lambda$-orbit is
dense, so its minimal orbit closure is all of $X$.
\end{proof}

\begin{theorem}
\label{thm:mixed-product-rigidity}
\label{prop:cartesian-rigidity}
\label{prop:difference-rigidity}
Let $(X,\Gamma)$ be a minimal point-distal system, let
$R_1,\ldots,R_d\in\Gamma$, and put
$
 A:=R_1\times\cdots\times R_d
$
on $X^d$.  Then
\begin{enumerate}[label=\textup{(\roman*)}]
\item
$
 \Dist_\Gamma(X)^d\subseteq\Dist_A(X^d).
$
In particular, the set of $A$-minimal points contains a dense
$G_\delta$ subset of $X^d$.

\item If $A$ is topologically transitive, then $(X^d,A)$ is minimal
point-distal.

\item If the elements $R_1,\ldots,R_d$ commute and $A$ is topologically
transitive, then every difference transformation
$
 R_i^{-1}R_j,
$ $i\neq j,
$
is minimal point-distal.
\end{enumerate}
\end{theorem}

\begin{proof}
(i) is direct from the fact that $R_1,\ldots,R_d\in \Gamma$.

We now prove (ii). If $A$ is topologically transitive, the set of $A$-transitive points is a
dense $G_\delta$ subset of $X^d$.  Choose a point in its intersection with
$ \Dist_\Gamma(X)^d$.  This point is simultaneously $A$-transitive and $A$-distal. Hence
it is minimal, and its orbit closure is all of $X^d$.  This proves part~(ii).

Assume now that the $R_i$ commute.  By part~(ii), the projection of
$(X^d,A)$ to the $(i,j)$ coordinates makes $(X^2,R_i\times R_j)$ minimal.
Let $U,V\subseteq X$ be nonempty open sets and fix $x_0\in X$.  Choose
$n\in\mathbb Z$ such that
$
 R_i^nx_0\in U$
 and $ R_j^nx_0\in V.
$
For $u=R_i^nx_0$, commutativity gives
\[
 (R_i^{-1}R_j)^nu=R_j^nx_0\in V.
\]
Thus $R_i^{-1}R_j$ is transitive.  It generates a subgroup of $\Gamma$, so
Lemma~\ref{lem:subgroup-rigidity} makes it minimal point-distal.
\end{proof}

\begin{corollary}
\label{cor:two-homeomorphisms-rigidity}
Let $T$ and $S$ be commuting minimal point-distal homeomorphisms of $X$.
\begin{enumerate}[label=\textup{(\roman*)}]
\item For every $(a,b)\in\mathbb Z^2$, if $T^aS^b$ is transitive, then it is
minimal point-distal.
\item If $T\times S$ is transitive on $X^2$, then both $T\times S$ and
$T^{-1}S$ are minimal point-distal.
\end{enumerate}
\end{corollary}

\begin{proof}
By Theorem~\ref{thm:intro-joint}, the $\langle T,S\rangle$-action is minimal
point-distal.  Part~(i) is Lemma~\ref{lem:subgroup-rigidity}, and part~(ii)
follows from Theorem~\ref{thm:mixed-product-rigidity}.
\end{proof}

\subsection{Linear iterates}
\label{sec:linear}
We now combine the mixed-product rigidity theorem with the
linear-iterate criterion of Donoso--Koutsogiannis--Sun \cite{DKS}.  In a general
minimal $\mathbb Z^k$-system, joint transitivity requires both
transitivity of the product transformation and transitivity of all
pairwise differences.  In the point-distal category, the preceding
results show that the difference conditions are automatic once the
product is transitive.

Put
$\Gamma=\langle S_1,\ldots,S_k\rangle$ and let $(X,\Gamma)$ be a minimal $\mathbb Z^k$-system.  For
$T_1,\ldots,T_d\in\Gamma$, set
\[
 A:=T_1\times\cdots\times T_d
 \quad\text{on }X^d,
 \qquad
 \Delta_d(x):=(x,\ldots,x).
\]

\begin{definition}
The linear families $(T_1^n)_{n\in\mathbb Z},\ldots,(T_d^n)_{n\in\mathbb Z}$ are \emph{jointly
transitive} if there is a dense $G_\delta$ set $X_0\subseteq X$ such that
$
 \{A^n\Delta_d(x):n\in\mathbb Z\}
$
is dense in $X^d$ for every $x\in X_0$.  They are \emph{jointly minimal} if
this orbit is dense for every $x\in X$.
\end{definition}

Donoso, Koutsogiannis and Sun proved that in an arbitrary minimal
$\mathbb Z^k$-system the families $(T_1^n)_n,\ldots,(T_d^n)_n$ are jointly
transitive if and only if every $T_i^{-1}T_j$, $i\neq j$, is transitive and
$T_1\times\cdots\times T_d$ is transitive
\cite[Theorem~1.6]{DKS}.  In the point-distal category, the difference
conditions are automatic and all transitivity conclusions upgrade to
minimality.

\begin{theorem}
\label{thm:linear-rigidity}
Let $(X,\Gamma)$ be a minimal point-distal $\mathbb Z^k$-system, and
let $T_1,\ldots,T_d\in\Gamma$.  The following are equivalent:
\begin{enumerate}[label=\textup{(\roman*)}]
\item $T_1\times\cdots\times T_d$ is topologically transitive on $X^d$;
\item $(X^d,T_1\times\cdots\times T_d)$ is minimal point-distal;
\item $(T_1^n)_n,\ldots,(T_d^n)_n$ are jointly transitive;
\item $(T_1^n)_n,\ldots,(T_d^n)_n$ are jointly minimal.
\end{enumerate}
Whenever these conditions hold, every difference
$
 T_i^{-1}T_j,
$ $i\neq j,
$
is minimal point-distal.  
\end{theorem}

\begin{proof}
Theorem~\ref{thm:mixed-product-rigidity} gives (i)$\Rightarrow$(ii).  Under
(ii), every point of $X^d$ has dense orbit, and hence every diagonal point
$\Delta_d(x)$ has dense orbit; thus (ii)$\Rightarrow$(iv).  The implication
(iv)$\Rightarrow$(iii) is immediate.  Joint transitivity provides a
transitive diagonal point, so (iii)$\Rightarrow$(i).  The final assertion
follows from part~(iii) of Theorem~\ref{thm:mixed-product-rigidity}.
\end{proof}

\begin{remark}[Affine linear iterates]
\label{rem:affine-linear}
Let $a_i\in\mathbb Z\setminus\{0\}$ and $b_i\in\mathbb Z$.  Applying
Theorem~\ref{thm:linear-rigidity} to $T_i^{a_i}$ gives the corresponding
criterion for
\[
 (T_1^{a_1n+b_1})_n,\ldots,(T_d^{a_dn+b_d})_n.
\]
Indeed, the fixed coordinatewise homeomorphism
$T_1^{b_1}\times\cdots\times T_d^{b_d}$ preserves density of the relevant
orbit sets.
\end{remark}

\section{Canonical factors and joint towers}
\label{sec:canonical}

The preceding sections establish the existence of a joint point-distal
structure and its orbit-theoretic consequences.  We now turn to the finer
question of canonical compatibility.  The issue is local: over a fixed
common factor, one must compare the maximal highly proximal and maximal
isometric intermediate factors for the two subactions and for the action
they generate.  Once this local comparison is understood, it can be iterated
along canonical HPI or Furstenberg towers.

\subsection{Comparison over a common factor}
\label{subsec:comparison-common-factor}

We begin by fixing the lattice convention for intermediate factors.

Let
$
 q:X\to Y
$
be a factor map. An \emph{intermediate factor} of $q$ is a factorization
\[
 X\xrightarrow{p}Z\xrightarrow{\pi}Y,
 \qquad
 q=\pi p.
\]

If
\[
 X\xrightarrow{p_i}Z_i\xrightarrow{\pi_i}Y,
 \qquad i=1,2,
\]
are intermediate factors, we write
$
 Z_1\preceq_Y Z_2
$
if $Z_1$ is a factor of $Z_2$ over $Y$. Equivalently,
$
 R_{p_2}\subseteq R_{p_1}.
$

The meet
$
 Z_1\wedge_Y Z_2
$
is the greatest intermediate factor which is a factor of both $Z_1$ and
$Z_2$. Its kernel relation is
\[
 R_{Z_1\wedge_Y Z_2}
 =
 R_{p_1}\vee R_{p_2},
\]
where the right-hand side denotes the smallest closed equivalence relation
containing $R_{p_1}\cup R_{p_2}$.

To compare isometric successors, we first note that relative
equicontinuity for the two commuting subactions combines to relative
equicontinuity for the generated action.

\begin{lemma}
\label{lem:separate-joint-equicontinuity}
Let
$
 \pi:(Z,\Gamma)\to(Y,\Gamma)
$
be a factor map. If $\pi$ is relatively equicontinuous for both
$G$ and $H$, then it is relatively equicontinuous for
$\Gamma=\langle G,H\rangle$. Consequently, if $\pi$ is both
$G$-isometric and $H$-isometric, then it is $\Gamma$-isometric.
\end{lemma}

\begin{proof}
Fix a compatible metric $d$ on $Z$ and let $\varepsilon>0$.
By relative equicontinuity for $G$, choose $\eta>0$ such that
\[
 \pi(u)=\pi(v),\quad d(u,v)<\eta
 \quad\Longrightarrow\quad
 d(gu,gv)<\varepsilon
 \qquad(g\in G).
\]
By relative equicontinuity for $H$, there exists $\delta>0$ such that
\[
 \pi(u)=\pi(v),\quad d(u,v)<\delta
 \quad\Longrightarrow\quad
 d(hu,hv)<\eta
 \qquad(h\in H).
\]
Since the actions commute, every element of
$\Gamma=\langle G,H\rangle$ has the form $gh$, with
$g\in G$ and $h\in H$. Hence
\[
 \pi(u)=\pi(v),\quad d(u,v)<\delta
 \quad\Longrightarrow\quad
 d(ghu,ghv)<\varepsilon
\]
for every $g\in G$ and $h\in H$. Thus $\pi$ is relatively
equicontinuous for $\Gamma$.
\end{proof}

The identification of the common maximal equicontinuous factor uses one
additional elementary fact about compact homogeneous spaces.

\begin{lemma}
\label{lem:compact-centralizer}
Let a compact group $K$ act continuously and transitively on a compact
Hausdorff space $Y$. Then the centralizer
\[
 C_{\Homeo(Y)}(K)
 :=
 \{\varphi\in\Homeo(Y):\varphi k=k\varphi
   \text{ for every }k\in K\}
\]
is a compact group. More precisely, if $y_0\in Y$ and
\[
 L:=K_{y_0}
   =\{k\in K:ky_0=y_0\},
\]
then
\[
 C_{\Homeo(Y)}(K)\cong N_K(L)/L,
\]
where
$
 N_K(L):=\{a\in K:aLa^{-1}=L\}
$
is the normalizer of $L$ in $K$. In particular, every subgroup of
$C_{\Homeo(Y)}(K)$ acts equicontinuously on $Y$.
\end{lemma}

\begin{proof}
Since the action is continuous and $Y$ is Hausdorff, $L$ is a closed
subgroup of $K$. Consider the orbit map
\[
 \theta:K\longrightarrow Y,
 \qquad
 \theta(k):=ky_0.
\]
The map $\theta$ is continuous and surjective, as the $K$-action
is transitive.

For $k_1,k_2\in K$, one has
\[
 \theta(k_1)=\theta(k_2)
\qquad\text{if and only if}\qquad
 k_1y_0=k_2y_0,
\]
which is equivalent to
$
 k_2^{-1}k_1\in L,
$
and hence to
$
 k_1L=k_2L.
$
Thus the fibres of $\theta$ are precisely the left cosets of $L$.
Therefore $\theta$ induces a continuous bijection
\[
 \overline\theta:K/L\longrightarrow Y,
 \qquad
 \overline\theta(kL)=ky_0.
\]
Since $L$ is closed and $K$ is compact, the quotient $K/L$ is compact.
As $Y$ is Hausdorff, $\overline\theta$ is a homeomorphism. It is also
$K$-equivariant with respect to the left translation action on $K/L$,
since
\[
 \overline\theta(akL)
 =
 aky_0
 =
 a\,\overline\theta(kL).
\]
Hence we may identify $Y$ with the homogeneous space $K/L$.

We next determine the homeomorphisms of $K/L$ which commute with the
left $K$-action. Let
$
 \varphi\in C_{\Homeo(K/L)}(K).
$
Since $\varphi$ is $K$-equivariant, it is completely determined by the
point $\varphi(L)$. Write
$
 \varphi(L)=aL
$
for some $a\in K$. Therefore
\[
 aL
 =
 \varphi(L)
 =
 \varphi(\ell L)
 =
 \ell\varphi(L)
 =
 \ell aL,
\]
so
\[
 a^{-1}La\subseteq L.
\]
Applying the same argument to $\varphi^{-1}$ gives the reverse inclusion,
and therefore
\[
 a^{-1}La=L.
\]
Thus $a\in N_K(L)$.

Conversely, if $a\in N_K(L)$, define
\[
R_a:K/L\longrightarrow K/L,
\qquad
R_a(kL):=ka^{-1}L.
\]
This is well defined. Indeed, if $k_1L=k_2L$, then
$k_2^{-1}k_1\in L$, and since $a\in N_K(L)$,
$
a(k_2^{-1}k_1)a^{-1}\in L.
$
Equivalently,
$
k_1a^{-1}L=k_2a^{-1}L.
$

The map $R_a$ is a homeomorphism, with inverse $R_{a^{-1}}$, and it
commutes with the left $K$-action. Moreover,
$
R_a\circ R_b=R_{ab}.
$
Thus
\[
N_K(L)\longrightarrow C_{\Homeo(K/L)}(K),
\qquad
a\longmapsto R_a,
\]
is a group homomorphism. Its kernel is $L$, and every element of
$C_{\Homeo(K/L)}(K)$ is of this form. Hence
\[
N_K(L)/L\cong C_{\Homeo(K/L)}(K).
\]
Transporting this identification through
$\overline\theta:K/L\to Y$ gives
\[
 C_{\Homeo(Y)}(K)\cong N_K(L)/L.
\]

Since $N_K(L)$ is a closed subgroup of the compact group $K$,
the quotient $N_K(L)/L$ is compact. Hence the centralizer
$C_{\Homeo(Y)}(K)$ is compact in the compact-open topology.

Finally, a compact subgroup of $\Homeo(Y)$ acts equicontinuously on the
compact Hausdorff space $Y$. Therefore every subgroup of
$C_{\Homeo(Y)}(K)$ acts equicontinuously on $Y$.
\end{proof}

We can now compare the three canonical operations in a single statement.

\begin{theorem}
\label{thm:canonical-factor-comparison}
The following hold.
\begin{enumerate}[label=\textup{(\roman*)}]
\item The three actions have the same maximal equicontinuous factor as a
quotient of $X$
\[
 \MEF(X,G)=\MEF(X,H)=\MEF(X,\Gamma).
\]

\item For every common factor map
$
 q:(X,\Gamma)\to(Y,\Gamma),
$
the canonical maximal highly proximal intermediate factors coincide:
\[
 \mathsf H_G(q)=\mathsf H_H(q)=\mathsf H_\Gamma(q).
\]

\item For every such common factor,
\[
 \mathsf I_\Gamma(q)
 =\mathsf I_G(q)\wedge_Y\mathsf I_H(q).
\]
Equivalently,
\[
 E_\Gamma(q)=E_G(q)\vee E_H(q).
\]
\end{enumerate}
\end{theorem}
\begin{proof}
We first prove part~(i).  Let
\[
 \pi_G:X\longrightarrow Y_G:=\MEF(X,G).
\]
Since $Y_G$ is the maximal $G$-isometric factor of the map $X\to\{*\}$,
Proposition~\ref{prop:characteristic-successors} implies that the $H$-action
descends to $Y_G$.  Let
\[
 K:=\overline{G}^{\,\Homeo(Y_G)}.
\]
As the $G$-action on $Y_G$ is minimal and equicontinuous, $K$ is a compact
group acting transitively on $Y_G$.  The induced $H$-action commutes with $G$
and hence with $K$.  By Lemma~\ref{lem:compact-centralizer}, the image of $H$
in $\Homeo(Y_G)$ is equicontinuous.  Thus $Y_G$ is also an
$H$-equicontinuous factor of $X$, and maximality gives
\[
 Y_G\preceq \MEF(X,H).
\]
Interchanging $G$ and $H$ yields the reverse comparison.  Hence
\[
 \MEF(X,G)=\MEF(X,H)=:Y_0.
\]
The factor $Y_0$ is equicontinuous for both $G$ and $H$, and therefore for
$\Gamma$ by Lemma~\ref{lem:separate-joint-equicontinuity}, applied over the
one-point system.  Hence $Y_0\preceq\MEF(X,\Gamma)$.  Conversely, every
$\Gamma$-equicontinuous factor is $G$-equicontinuous, so
$\MEF(X,\Gamma)\preceq Y_0$.  This proves part~(i).

For part~(ii), put
\[
 Z_G:=\mathsf H_G(q).
\]
By Proposition~\ref{prop:characteristic-successors}, $Z_G$ is a common
$\Gamma$-factor.  The map $Z_G\to Y$ is irreducible.  Since irreducibility is a
property of the underlying continuous surjection, the same map is highly
proximal for the $H$-action and for the joint $\Gamma$-action.  Therefore
\[
 Z_G\preceq_Y\mathsf H_H(q),
 \qquad
 Z_G\preceq_Y\mathsf H_\Gamma(q).
\]
Interchanging $G$ and $H$ gives
\[
 \mathsf H_H(q)\preceq_Y Z_G,
\]
and hence $\mathsf H_G(q)=\mathsf H_H(q)$.  Moreover,
$\mathsf H_\Gamma(q)\to Y$ is irreducible and therefore $G$-highly proximal.
The maximality of $\mathsf H_G(q)$ gives
\[
 \mathsf H_\Gamma(q)\preceq_Y\mathsf H_G(q),
\]
which completes the proof of part~(ii).

For part~(iii), every $\Gamma$-isometric intermediate factor is both
$G$-isometric and $H$-isometric.  Consequently,
\[
 \mathsf I_\Gamma(q)
 \preceq_Y
 M:=\mathsf I_G(q)\wedge_Y\mathsf I_H(q).
\]
By Proposition~\ref{prop:characteristic-successors}, the two maximal
subaction-isometric factors are common $\Gamma$-factors, and hence so is $M$.
The extension $M\to Y$ is a factor over $Y$ of $\mathsf I_G(q)\to Y$, and is
therefore $G$-isometric.  Similarly it is $H$-isometric.  Lemma
\ref{lem:separate-joint-equicontinuity} makes it $\Gamma$-isometric, and the
maximality of $\mathsf I_\Gamma(q)$ gives
\[
 M\preceq_Y\mathsf I_\Gamma(q).
\]
Thus
\[
 \mathsf I_\Gamma(q)
 =\mathsf I_G(q)\wedge_Y\mathsf I_H(q).
\]
Since $\mathsf I_L(q)=X/E_L(q)$ and the correspondence between intermediate
factors and their kernel equivalence relations reverses the lattice order, the
relation identity follows.
\end{proof}

The local comparison immediately extends to any finite family of pairwise
commuting minimal actions.

\begin{corollary}
\label{cor:finite-canonical-comparison}
Let $G_1,\ldots,G_r$
act pairwise commutatively on a compact metric sapce $X$, and suppose that each $G_i$-action is minimal and distal.
Let
$
 \Gamma_r:=\langle G_1,\ldots,G_r\rangle,
$
and let
$
 q:(X,\Gamma_r)\to(Y,\Gamma_r)
$
be a factor map. Then
\[
 \mathsf H_{G_1}(q)
 =
 \cdots
 =
 \mathsf H_{G_r}(q)
 =
 \mathsf H_{\Gamma_r}(q)
\qquad\text{and}\qquad
 \mathsf I_{\Gamma_r}(q)
 =
 \bigwedge_{i=1}^r\mathsf I_{G_i}(q).
\]
\end{corollary}

More importantly for canonical towers, the same theorem identifies exactly
where two synchronized HPI constructions may first diverge.

\begin{corollary}
\label{cor:first-canonical-discrepancy}
Let $G$ and $H$ be commuting minimal strictly HPI actions on $X$, and
write their canonical constructions in synchronized expanded form, starting
from their common maximal equicontinuous factor. Suppose that the two
constructions agree through a common stage
$
q_\alpha:X\to Y_\alpha.
$
Then a highly proximal successor of $Y_\alpha$ is the same for the
$G$- and $H$-actions, and is also the maximal highly proximal successor for
the joint $\Gamma$-action.

Consequently, the first discrepancy between the two canonical towers can
occur only at an isometric stage. If
$Y_{\alpha+1}^{G}$ and $Y_{\alpha+1}^{H}$ are the corresponding maximal
isometric successors over $Y_\alpha$, then the maximal $\Gamma$-isometric
successor is
\[
Y_{\alpha+1}^{\Gamma}
=
Y_{\alpha+1}^{G}\wedge_{Y_\alpha}Y_{\alpha+1}^{H}.
\]
\end{corollary}

\subsection{The joint Furstenberg tower}

In the distal category there are no highly proximal successor stages.  The
local meet formula can therefore be iterated, giving both joint distality and
a canonical recursive description of the generated action.

\begin{theorem}[Joint distality and the joint Furstenberg tower]
\label{thm:joint-distal}
Assume that both $(X,G)$ and $(X,H)$ are minimal distal.  Then
$(X,\Gamma)$ is minimal distal.

Moreover, the canonical Furstenberg tower of the joint action is obtained
recursively as follows.  Let
\[
 Y_0:=\MEF(X,G)=\MEF(X,H)=\MEF(X,\Gamma).
\]
If
$
 q_\alpha:X\to Y_\alpha
$
is the current joint canonical factor, define
\[
 Y_{\alpha+1}
 :=\mathsf I_G(q_\alpha)\wedge_{Y_\alpha}\mathsf I_H(q_\alpha),
\]
and at a limit ordinal $\lambda$ put
\[
 Y_\lambda:=\varprojlim_{\alpha<\lambda}Y_\alpha.
\]
After deleting identity stages, this is the canonical Furstenberg tower of
$(X,\Gamma)$.
\end{theorem}

\begin{proof}
Consider the canonical Furstenberg distal tower of $(X,G)$:
\[
 \{*\}=X_0
 \longleftarrow X_1
 \longleftarrow\cdots
 \longleftarrow X_\alpha
 \longleftarrow\cdots
 \longleftarrow X_\eta=X.
\]
Every successor map is $G$-isometric and every limit stage is an inverse
limit.  Corollary~\ref{cor:canonical-characteristic} makes every $X_\alpha$ a
common $G,H$-factor.  Since distality passes to factors, every point of every
$(X_\alpha,H)$ is $H$-distal.

We prove by transfinite induction that every point of $X_\alpha$ is
$\Gamma$-distal.  The assertion is trivial at $X_0$.  Suppose that
$
 \pi_\alpha:X_{\alpha+1}\to X_\alpha
$
is an isometric successor and that every point of $X_\alpha$ is
$\Gamma$-distal.  For $x\in X_{\alpha+1}$, the point
$\pi_\alpha(x)$ is $\Gamma$-distal and $x$ is $H$-distal.
Lemma~\ref{lem:mixed-isometric} shows that $x$ is $\Gamma$-distal.  At a
limit ordinal, two $\Gamma$-proximal points have equal projections to every
preceding stage and hence are equal.  Thus every point of $X$ is
$\Gamma$-distal.  Minimality is automatic.

The joint action is therefore distal, so its canonical Furstenberg tower
consists only of maximal isometric successor operations and inverse limits.
At every current common factor $q_\alpha:X\to Y_\alpha$, part~(iii) of
Theorem~\ref{thm:canonical-factor-comparison} identifies the maximal
$\Gamma$-isometric successor with
\[
 \mathsf I_G(q_\alpha)\wedge_{Y_\alpha}\mathsf I_H(q_\alpha).
\]
The limit prescription is the canonical inverse-limit construction.
\end{proof}
Applying Theorem~\ref{thm:joint-distal} inductively, we immediately obtain the following.
\begin{corollary}
\label{cor:finite-distal-families}
Let $G_1,\ldots,G_r$ act pairwise commutatively and minimally distally on
$X$, and put $\Gamma_r:=\langle G_1,\ldots,G_r\rangle$.  Then the joint
$\Gamma_r$-action is minimal distal.  At every successor stage of its
canonical Furstenberg tower, the maximal isometric factor is
$
 \bigwedge_{i=1}^r\mathsf I_{G_i}(q_\alpha).
$
\end{corollary}

The meet formula thus produces the joint canonical tower, but it does not
force either subaction tower to coincide with it.  The next section shows
that the meet can be proper even for commuting minimal distal
homeomorphisms.

\section{Canonical non-rigidity: a distal counterexample}
\label{sec:counterexample}

We now show that the meet appearing in
Theorem~\ref{thm:canonical-factor-comparison} can be strict.  More
precisely, we construct commuting minimal distal homeomorphisms $T$ and
$S$ of $\T^3$ which have the same maximal equicontinuous factor
$Y=\T$, but whose maximal isometric factors over $Y$ are different:
\[
   \mathsf I_{\langle T\rangle}(Y)=\T^3,
   \qquad
   \mathsf I_{\langle S\rangle}(Y)
   =
   \mathsf I_{\langle T,S\rangle}(Y)
   =
   \T^2.
\]
The construction combines a small-divisor argument, used to guarantee
minimality and to identify the maximal equicontinuous factor, with a fibre
shear, which determines the relative equicontinuous structure relation for
the $S$-action.

\subsection{Construction and minimality}

We first isolate the two cohomological tools used to choose the cocycles.

Write
\[
 \e(t):=\exp(2\pi i t),
 \qquad t\in\R,
\]
and identify real-valued cocycles with their classes modulo $1$ when they
occur in torus skew products.  For $h\in C(\T,\mathbb C)$, use the convention
\[
 \widehat h(k)=\int_\T h(t)\e(-kt)\,dt.
\]

\begin{lemma}\label{lem:fourier-obstruction}
Let $\alpha\in\T$ be irrational and let $h\in C(\T,\R)$.  Suppose that there
exist $\zeta\in\T$ and $H\in C(\T,\T)$ such that
\begin{equation}\label{eq:circle-cohomology}
 \e(h(t))=\zeta\frac{H(t+\alpha)}{H(t)}
 \qquad(t\in\T).
\end{equation}
Then there exists $u\in C(\T,\R)$ such that, for every nonzero $k\in\Z$,
\begin{equation}\label{eq:fourier-transfer}
 \widehat h(k)=\bigl(\e(k\alpha)-1\bigr)\widehat u(k).
\end{equation}
Consequently, if $|k_j|\to\infty$, then
\begin{equation}\label{eq:fourier-ratio-zero}
 \frac{\widehat h(k_j)}{\e(k_j\alpha)-1}\longrightarrow0.
\end{equation}
\end{lemma}

\begin{proof}
Write $H(t)=\e(dt+u(t))$ for some $d\in\Z$ and $u\in C(\T,\R)$, and write
$\zeta=\e(c)$.  Equation~\eqref{eq:circle-cohomology} implies that
\[
 h(t)-c-d\alpha-u(t+\alpha)+u(t)
\]
is a continuous integer-valued function and hence constant.  Comparing
nonzero Fourier coefficients gives \eqref{eq:fourier-transfer}.  The final
claim is the Riemann--Lebesgue lemma.
\end{proof}

We use the following standard compact-group extension criterion
\cite[Theorem~1 and Corollary~2]{Parry1969}.

\begin{lemma}\label{lem:parry}
Let $R:Y\to Y$ be minimal, let $K$ be a compact metrizable abelian group, and
let $c:Y\to K$ be continuous.  The skew product
$
 R_c(y,k)=(Ry,k+c(y))
$
is minimal if and only if there do not exist a nontrivial character
$\chi\in\widehat K$ and $H\in C(Y,\T)$ such that
$
 \chi(c(y))=\frac{H(Ry)}{H(y)}
$ for all $y\in Y$.
\end{lemma}

We now choose the rotation parameters and cocycles.  The small-divisor
estimates are designed simultaneously to force minimality and to make the
two skew products commute.

For $j\geq1$, set
$
 k_j:=2^{2j(j-1)}
$
and define
\begin{equation}\label{eq:tau-sigma}
 \tau:=\sum_{r=1}^\infty2^{-2r(r+1)},
 \qquad
 \sigma:=\sum_{r=1}^\infty2^{-2r^2}.
\end{equation}
Put
\begin{equation}\label{eq:lambda-mu}
 \lambda_j:=\e(k_j\tau)-1,
 \qquad
 \mu_j:=\e(k_j\sigma)-1.
\end{equation}

\begin{lemma}[Small divisors]\label{lem:small-divisors}
The numbers $\tau$ and $\sigma$ are irrational, and as $j\to\infty$,
\begin{equation}\label{eq:small-divisor-asymptotics}
 |\lambda_j|\asymp2^{-4j},
 \qquad
 |\mu_j|\asymp2^{-2j}.
\end{equation}
In particular,
\begin{equation}\label{eq:bounded-small-divisor-ratio}
 \sup_{j\geq1}\left|\frac{\mu_j^2}{\lambda_j}\right|<\infty.
\end{equation}
\end{lemma}

\begin{proof}
The binary expansions in \eqref{eq:tau-sigma} have nonzero digits at
positions $2r(r+1)$ and $2r^2$, respectively.  The gaps between successive
nonzero digits tend to infinity, and neither expansion is eventually
periodic.  Hence neither number is rational.

For $r\leq j-1$, the quantity
$2^{2j(j-1)-2r(r+1)}$ is an integer.  Therefore, modulo integers,
\[
 k_j\tau=2^{-4j}
 +\sum_{r=j+1}^\infty2^{2j(j-1)-2r(r+1)}
 =2^{-4j}+O(2^{-8j}).
\]
Similarly,
\[
 k_j\sigma=2^{-2j}
 +\sum_{r=j+1}^\infty2^{2j(j-1)-2r^2}
 =2^{-2j}+O(2^{-6j}).
\]
Since $|\e(s)-1|=2|\sin(\pi s)|\asymp|s|$ as $s\to0$,
\eqref{eq:small-divisor-asymptotics} follows.  The boundedness in
\eqref{eq:bounded-small-divisor-ratio} follows after enlarging the bound to
cover finitely many initial indices.
\end{proof}

Let $a_j=2^{-j}$ and define real-valued continuous functions on $\T$ by
\begin{align}
 f(t):=2\operatorname{Re}\sum_{j=1}^\infty a_j\e(k_jt),
    \quad                                                    
 g(t):=f(t+\sigma)-f(t),      \quad                        
 v(t):=2\operatorname{Re}\sum_{j=1}^\infty
       a_j\frac{\mu_j^2}{\lambda_j}\e(k_jt).          
\end{align}
Both displayed series converge absolutely and uniformly.  At the positive
frequency $k_j$,
\begin{equation}\label{eq:basic-Fourier-coefficients}
 \widehat f(k_j)=a_j,
 \qquad
 \widehat g(k_j)=a_j\mu_j,
 \qquad
 \widehat v(k_j)=a_j\frac{\mu_j^2}{\lambda_j}.
\end{equation}
Comparison of Fourier coefficients gives
\begin{equation}\label{eq:cocycle-commutation}
 v(t+\tau)-v(t)=g(t+\sigma)-g(t).
\end{equation}

Set $X=\T^3$, with coordinates $(t,x,z)$, and define
\begin{align}
 T(t,x,z)=(t+\tau,\ x+f(t),\ z+g(t))\quad               \text{and}
\quad S(t,x,z)=(t+\sigma,\ x+z,\ z+v(t)).
\end{align}
All coordinates are taken modulo $1$.

\begin{lemma}\label{lem:commutation-example}
The homeomorphisms $T$ and $S$ commute.
\end{lemma}

\begin{proof}
The identities $f(t+\sigma)=f(t)+g(t)$ and
\eqref{eq:cocycle-commutation} give
$
 TS(t,x,z)=ST(t,x,z)
$
by direct comparison of the three coordinates.
\end{proof}

Put
\[
 Y:=\T
 \qquad\text{and}\qquad
 Z:=\T^2
\]
with coordinates $(t,z)$ on $Z$, and define
\[
 \pi(t,x,z)=t,
 \qquad
 \rho(t,x,z)=(t,z),
 \qquad
 \pi_0(t,z)=t.
\]
The induced maps on $Z$ are
\begin{align*}
 T_1(t,z)=(t+\tau,z+g(t))\quad\text{and}\quad
 S_1(t,z)=(t+\sigma,z+v(t)).
\end{align*}

\begin{lemma}\label{lem:cohomological-obstructions}
The following hold.
\begin{enumerate}[label=\textup{(\roman*)}]
\item If $(m,n)\in\Z^2\setminus\{(0,0)\}$, there are no
      $\zeta\in\T$ and $H\in C(\T,\T)$ satisfying
      \begin{equation}\label{eq:T-general-obstruction}
       \e(mf(t)+ng(t))
       =\zeta\frac{H(t+\tau)}{H(t)}.
      \end{equation}
\item If $n\in\Z\setminus\{0\}$, there are no
      $\zeta\in\T$ and $H\in C(\T,\T)$ satisfying
      \begin{equation}\label{eq:S1-general-obstruction}
       \e(nv(t))
       =\zeta\frac{H(t+\sigma)}{H(t)}.
      \end{equation}
\end{enumerate}
\end{lemma}

\begin{proof}
At frequency $k_j$,
\[
 \widehat{(mf+ng)}(k_j)=a_j(m+n\mu_j).
\]
If $m\neq0$, then $|m+n\mu_j|$ is bounded away from zero for all sufficiently
large $j$, and
\[
 \left|\frac{a_j(m+n\mu_j)}{\lambda_j}\right|
 \asymp2^{3j}\longrightarrow\infty.
\]
If $m=0$, then $n\neq0$ and
\[
 \left|\frac{na_j\mu_j}{\lambda_j}\right|
 \asymp2^j\longrightarrow\infty.
\]
Both conclusions contradict Lemma~\ref{lem:fourier-obstruction}.  This proves
(i).  For (ii), the corresponding quotient is
\[
 \frac{n\widehat v(k_j)}{\mu_j}
 =na_j\frac{\mu_j}{\lambda_j},
\]
whose modulus is comparable to $2^j$, again a contradiction.
\end{proof}

With the commuting skew products in place, we verify minimality and record
the common, noncanonical tower of joint isometric extensions.

\begin{proposition}\label{prop:example-minimality}
The systems $(X,T)$, $(Z,S_1)$, and $(X,S)$ are minimal.
\end{proposition}

\begin{proof}
View $T$ as the $\T^2$-extension of the irrational rotation
$R_\tau:t\mapsto t+\tau$ with cocycle $(f,g)$.  A nontrivial character of
$\T^2$ has the form $(x,z)\mapsto\e(mx+nz)$ with $(m,n)\neq(0,0)$.
Lemma~\ref{lem:cohomological-obstructions}(i), in particular its case
$\zeta=1$, excludes every obstruction in Lemma~\ref{lem:parry}.  Hence $T$ is
minimal.

Similarly, $S_1$ is the circle extension of the irrational rotation
$R_\sigma$ with cocycle $v$.  Lemma~\ref{lem:cohomological-obstructions}(ii)
and Lemma~\ref{lem:parry} show that $S_1$ is minimal.

Finally, regard $S$ as the circle extension of $S_1$ with fibre coordinate
$x$ and cocycle $(t,z)\mapsto z$.  If $S$ were not minimal,
Lemma~\ref{lem:parry} would give $p\in\Z\setminus\{0\}$ and
$H\in C(Z,\T)$ such that
\begin{equation}\label{eq:S-second-obstruction}
 \e(pz)=\frac{H(t+\sigma,z+v(t))}{H(t,z)}.
\end{equation}
The integer $\deg_zH(t,\cdot)$ is independent of $t$.  The right-hand side of
\eqref{eq:S-second-obstruction} has degree zero in $z$, whereas the left-hand
side has degree $p\neq0$, a contradiction.
\end{proof}

\begin{proposition}
\label{prop:joint-isometric-tower}
Let $\Gamma=\langle T,S\rangle$.  The tower
\[
 \{*\}\longleftarrow Y
 \xleftarrow{\pi_0}Z
 \xleftarrow{\rho}X
\]
is a tower of $\Gamma$-isometric extensions.  Consequently, the joint
$\Gamma$-system is minimal distal, and both cyclic systems are minimal distal
and finite-step strictly PI.
\end{proposition}

\begin{proof}
The action on $Y$ is generated by the rotations $t\mapsto t+\tau$ and
$t\mapsto t+\sigma$ and is therefore equicontinuous.  On each fibre of
$\pi_0$, the standard circle metric
$
 d_\T(z,z')
$
is invariant under both $T_1$ and $S_1$, because the two points receive the
same cocycle increment.  On each fibre of $\rho$, the standard metric
$d_\T(x,x')$ is invariant under both $T$ and $S$ for the same reason.  Thus
both extensions are $\Gamma$-isometric.  The joint action is minimal because
it contains the minimal $T$-action, and an iterated isometric extension of a
distal flow is distal.  Restricting the action gives the assertions for $T$
and $S$.
\end{proof}

\subsection{The maximal equicontinuous factor and the shear}

We next identify the first canonical stage.  This separates the common
maximal equicontinuous factor from the later discrepancy in the isometric
successor.

\begin{lemma}\label{lem:T-eigenfunctions}
Every continuous eigenfunction of $(X,T)$ factors through
$\pi(t,x,z)=t$.
\end{lemma}

\begin{proof}
For $(u,w)\in\T^2$, let
\[
 V_{u,w}(t,x,z)=(t,x+u,z+w).
\]
These vertical translations commute with $T$.  Let $F$ be a nonzero
continuous $T$-eigenfunction.  Minimality makes $|F|$ constant, so after
normalization $F:X\to\T$.  For every $(u,w)$, the function $F\circ V_{u,w}$
is a $T$-eigenfunction with the same eigenvalue.  The corresponding
eigenspace is one-dimensional, since the quotient of two such eigenfunctions
is a continuous $T$-invariant function.  Therefore
\[
 F\circ V_{u,w}=c(u,w)F
\]
for a continuous character $c:\T^2\to\T$.  Hence there are $m,n\in\Z$ and
$H\in C(\T,\T)$ such that
\begin{equation}\label{eq:eigenfunction-form}
 F(t,x,z)=H(t)\e(mx+nz).
\end{equation}
Substitution into the eigenvalue equation shows that $\e(mf+ng)$ is
continuously cohomologous to a constant over $R_\tau$.
Lemma~\ref{lem:cohomological-obstructions}(i) forces $(m,n)=(0,0)$, and $F$
depends only on $t$.
\end{proof}

\begin{proposition}\label{prop:common-mef}
The map $\pi:X\to Y$ is the maximal equicontinuous factor map for the
$T$-, $S$-, and joint $\Gamma$-actions.
\end{proposition}

\begin{proof}
The factor $(Y,R_\tau)$ is equicontinuous for the $T$-action, while
Lemma~\ref{lem:T-eigenfunctions} shows that every continuous $T$-eigenfunction
factors through $\pi$. Hence $\pi$ is the maximal equicontinuous factor map
of $(X,T)$.

Part~\textup{(i)} of Theorem~\ref{thm:canonical-factor-comparison} shows
that commuting minimal actions and their joint action have the same maximal
equicontinuous factor as a quotient of $X$. Therefore $\pi$ is also the
maximal equicontinuous factor map for the $S$- and $\Gamma$-actions.
\end{proof}
Having fixed the common maximal equicontinuous factor, we determine the
relative equicontinuous structure over it.  The next lemma isolates the fibre
shear as the precise obstruction to making $\pi$ isometric for $S$.

\begin{lemma}\label{lem:shear}
Let
$
 R_\rho=\{((t,x,z),(t,x',z)):t,x,x',z\in\T\}.
$
Then
$
 R_\rho\subseteq Q_{\langle S\rangle}(\pi).
$
\end{lemma}

\begin{proof}
Fix
$
 u=(t,x,z),
 u'=(t,x',z) \in R_\rho
$, and choose $a\in[-1/2,1/2]$ representing $x'-x\in\T$.  Let
$n_j\to\infty$ and put $\delta_j=-a/n_j$ modulo $1$.  Define
\[
 u_j=(t,x,z),
 \qquad
 u_j'=(t,x',z+\delta_j).
\]
Then $u_j\to u$, $u_j'\to u'$, and $\pi(u_j)=\pi(u_j')$.

For two points with the same $t$-coordinate, the difference of their
$z$-coordinates remains constant under iteration by $S$, while the shear
$x\mapsto x+z$ gives
\[
 \operatorname{pr}_x(S^ku_j')-
 \operatorname{pr}_x(S^ku_j)
 =a+k\delta_j
 \pmod1.
\]
At $k=n_j$ the $x$-coordinates agree and the $z$-difference equals
$\delta_j\to0$.  Hence
\[
 d(S^{n_j}u_j,S^{n_j}u_j')\longrightarrow0.
\]
Thus $(u,u')\in Q_{\langle S\rangle}(\pi)$.
\end{proof}

\begin{proposition}
\label{prop:exact-relative-factors}
One has
\[
 E_{\langle T\rangle}(\pi)=\Delta_X,
 \qquad
 E_{\langle S\rangle}(\pi)=E_\Gamma(\pi)=R_\rho.
\]
Consequently,
\[
 \mathsf I_{\langle T\rangle}(Y)=X,
 \qquad
 \mathsf I_{\langle S\rangle}(Y)
 =\mathsf I_\Gamma(Y)=Z.
\]
The canonical distal towers are therefore
\[
 \{*\}\longleftarrow Y\longleftarrow X
 \quad\text{for }T,
\]
and
\[
 \{*\}\longleftarrow Y\longleftarrow Z\longleftarrow X
 \quad\text{for }S\text{ and for }\Gamma.
\]
\end{proposition}

\begin{proof}
On every $\pi$-fibre, $T$ acts by the translation
\[
 (x,z)\longmapsto(x+f(t),z+g(t)).
\]
The standard product metric on $\T^2$ is therefore a continuous
$T$-invariant relative metric. Hence $\pi$ is $T$-isometric and
\[
 E_{\langle T\rangle}(\pi)=\Delta_X,
 \qquad
 \mathsf I_{\langle T\rangle}(Y)=X.
\]

The intermediate extension
$
 \pi_0:Z\to Y
$
is $S$-isometric by Proposition~\ref{prop:joint-isometric-tower}. Therefore
\[
 E_{\langle S\rangle}(\pi)\subseteq R_\rho.
\]
Lemma~\ref{lem:shear} gives
\[
 R_\rho
 \subseteq
 Q_{\langle S\rangle}(\pi)
 \subseteq
 E_{\langle S\rangle}(\pi).
\]
Consequently,
\[
 E_{\langle S\rangle}(\pi)=R_\rho,
 \qquad
 \mathsf I_{\langle S\rangle}(Y)=Z.
\]

Part~\textup{(iii)} of Theorem~\ref{thm:canonical-factor-comparison} now yields
\[
 \mathsf I_\Gamma(Y)
 =
 \mathsf I_{\langle T\rangle}(Y)
 \wedge_Y
 \mathsf I_{\langle S\rangle}(Y)
 =
 X\wedge_Y Z
 =
 Z.
\]
Equivalently,
\[
 E_\Gamma(\pi)
 =
 E_{\langle T\rangle}(\pi)
 \vee
 E_{\langle S\rangle}(\pi)
 =
 \Delta_X\vee R_\rho
 =
 R_\rho.
\]

Finally, Proposition~\ref{prop:common-mef} identifies $Y$ as the first
canonical stage, while $X\to Z$ is isometric for both $S$ and $\Gamma$.
Therefore the canonical distal towers are
\[
 \{*\}\longleftarrow Y\longleftarrow X
 \quad\text{for }T,
\]
and
\[
 \{*\}\longleftarrow Y\longleftarrow Z\longleftarrow X
 \quad\text{for }S\text{ and for }\Gamma.\qedhere
\]
\end{proof}

\begin{proof}[Proof of Theorem~\ref{thm:intro-counterexample}]
Commutation is Lemma~\ref{lem:commutation-example}; minimality and distality
are Propositions~\ref{prop:example-minimality} and
\ref{prop:joint-isometric-tower}; the common maximal equicontinuous factor is
Proposition~\ref{prop:common-mef}; and the exact canonical towers are computed
in Proposition~\ref{prop:exact-relative-factors}.
\end{proof}

\end{document}